\documentclass[11pt,a4paper]{article}

\usepackage{graphicx}
\usepackage{amsmath}
\usepackage{amsfonts}
\usepackage{amssymb}
\usepackage{enumerate}
\usepackage{lscape}
\usepackage{longtable}
\usepackage{rotating}
\usepackage[ruled,vlined,noline,linesnumbered]{algorithm2e}
\usepackage{algorithmic}
\usepackage{color}
\usepackage{url}
\usepackage{subfigure}
\usepackage{multirow}
\usepackage{rotating}
\newtheorem{theorem}{Theorem}[section]

\newtheorem{definition}{Definition}[section]

\newtheorem{lemma}{Lemma}[section]

\newenvironment{proof}[1][Proof]{\textbf{#1.} }{\ \rule{0.5em}{0.5em} \vspace{1ex}}

\def\real{\mathbb{R}}

\DeclareMathOperator{\es}{ES}
\DeclareMathOperator{\nr}{nr}
\DeclareMathOperator{\h}{H}
\DeclareMathOperator{\cl}{CL}
\DeclareMathOperator{\re}{r}
\DeclareMathOperator{\cmaes}{CMA-ES}
\DeclareMathOperator{\LB}{lb}
\DeclareMathOperator{\UB}{ub}

\DeclareMathOperator{\trial}{trial}

\begin{document}

\title{A Merit Function Approach for Evolution Strategies}
\author{
Youssef Diouane\thanks{Institut Sup\'erieur de l'A\'eronautique et de l'Espace (ISAE-SUPAERO), Universit\'e de Toulouse, 31055 Toulouse Cedex 4, France
 ({\tt youssef.diouane@isae.fr}).
}
}
\maketitle
\footnotesep=0.4cm
{\small
\begin{abstract}
In this paper, we extend a class of globally convergent evolution strategies to handle general constrained optimization problems. The proposed framework handles relaxable constraints using a merit function approach combined with a specific restoration procedure. The unrelaxable constraints in our framework, when present, are treated either by using the extreme barrier function or through a projection approach.

The introduced extension guaranties to the regarded class of evolution strategies global convergence properties for first order stationary constraints.
Preliminary numerical experiments are carried out on a set of known test problems as well as on a multidisciplinary design optimization problem. 
\end{abstract}

\bigskip

\begin{center}
\end{center}
}

\section{Introduction}
In this paper, we are interested by the following constrained optimization problem: 
\begin{equation} \label{c:Prob1}
\begin{split}
\min\quad & f(x)\\
\mbox{s.t.} \quad&   x \in \Omega = \Omega_{\re} \cap \Omega_{\nr},
\end{split}
\end{equation}
where the objective function $f$  is assumed to be locally Lipschitz continuous.
The feasible region $\Omega \subset \real^n $ of this problem includes two categories of constraints. The first one, denoted by $\Omega_{\re}$ and known as relaxable constraints (or soft constraints), is allowed to be violated during the optimization process and may need to be satisfied only approximately or asymptotically. Such set of constraints will be assumed, in the context of this paper, to be of the form:
$$
 \Omega_{\re} = \left\{ x \in \real^n |  \forall i \in \{1,\dots, r\} , c_i(x) \leq 0 \right\},$$ 
 where the functions $c_i$ are locally Lipschitz continuous.
The second category of constraints, denoted by $\Omega_{\nr} \subset \real^n$, gathers all unrelaxable constraints (also known as hard constraints), for such constraints no violation is allowed and they require satisfaction during all the optimization process. The set of constraints $\Omega_{\nr}$ can be seen as bounds or linear constraints. Many practical optimization problems present both relaxable and unrelaxable constraints. For instance, in multidisciplinary design optimization problems \cite{RPerez_etal_2004}, one has to deal with different coupled disciplines (e.g., structure, aerodynamics, propulsion) that represent the aircraft model. In this case, the appropriate design can be chosen by maximizing the aircraft range under bounds constraints and subject to many relaxable constraints composed of the involved disciplines.

Evolution strategies (ES's)~\cite{IRechenberg_1973} are one of the most successful stochastic optimization algorithms, seen as a class of 
evolutionary algorithms that are naturally parallelizable, appropriate for continuous optimization, and that lead to promising results on practical optimization problems~\cite{AAuger_NHansen_ZJPerez_RRos_MSchoenauer_2009,LMRios_NVSahinidis_2010,ZBouzarkouna_2012}. In~\cite{YDiouane_SGratton_LNVicente_2015_a,YDiouane_SGratton_LNVicente_2015_b}, the authors dealt with
a large class of ES's, where a certain number $\lambda$ of points (called offspring)
are randomly generated in each iteration, among which $\mu \leq \lambda$ of them (called parents) are
selected. ES's have been growing rapidly in popularity and start to be used for solving challenging 
optimization problems~\cite{NHansen_etal_2010,AAuger_DBrockhoff_NHansen_2013}.

In \cite{YDiouane_SGratton_LNVicente_2015_b}, the authors proposed a general globally convergent framework
for unrelaxable constraints using two different approaches.
The first one relies on techniques inspired from directional
direct-search methods~\cite{ARConn_KScheinberg_LNVicente_2009,TGKolda_RMLewis_VTorczon_2003},
where one uses an extreme barrier function
to prevent unfeasible displacements together with the possible use of
directions that conform to the local geometry of the feasible region.
The second approach was based  on enforcing all the generated
sample points to be feasible, all by using a projection mapping approach.
Both proposed strategies were compared to some of the best  available solvers for minimizing a function
without derivatives. The obtained numerical results confirmed
the competitiveness of the two approaches in terms of efficiency as well as robustness. 
Motivated by
the recent availability of massively parallel computing platforms, the authors in \cite{YDiouane_SGratton_XVasseur_LNVicente_HCalandra_2016} proposed a highly parallel globally convergent ES (inspired by \cite{YDiouane_SGratton_LNVicente_2015_b}) adapted to the full-waveform inversion setting.
By combining model reduction and ES's in a parallel environment, the authors helped solving realistic
instances of the  full-waveform inversion problem. 

For ES's enormous state of the art of constraints handling algorithms have been proposed \cite{constbib_Coello}. 
Coello~\cite{Coello_2002} and Kramer~\cite{Kramer_2010} outlined a 
comprehensive survey of the most popular handling constraints 
methods currently used with ES's. To the best of our knowledge, all the proposed ES's suffer from the lack of global convergence guarantees when applied to general constrained optimization problems.

In the framework of deterministic Derivative free optimization (DFO), only few works were interested to handle both kinds (relaxable and unrelaxable) of constraints separately. For instance,  Audet and 
Dennis~\cite{Audet_Dennis_2009} outlines a globally convergent  direct-search approach based on a progressive barrier, 
it combines an extreme barrier approach for unrelaxable constraints and
non-dominance filters~\cite{Fletcher_Leyffer_2002} to handle relaxable constraints.  More recently, the authors in \cite{CAudet_ARConn_SLeDigabel_MPeyrega_2018} extended the progressive barrier approach, developed in \cite{Audet_Dennis_2009}, to cover the setting of a derivative-free trust-region method. In the frame of directional
direct-search methods, Vicente and Gratton~\cite{SGratton_LNVicente_2014} 
proposed an alternative where one handles relaxable constraints by means of a merit function. 
The latter approach ensures global convergence by imposing a sufficient decrease condition on a merit function combining information from both objective function and constraints violation.  Another two-phases derivative-free approach is proposed in \cite{JMMartinez_FNCSobral_2013} to handle specifically the case where finding a 
feasible point is easier than minimizing the objective function.

 In this paper, inspired by the merit function approach for direct search methods~\cite{SGratton_LNVicente_2014}, we propose to adapt the class of ES algorithms proposed in \cite{YDiouane_SGratton_LNVicente_2015_b} to handle both relaxable and unrelaxable constraints.  
The obtained class of ES algorithms relies essentially on a merit function (eventually with a restoration procedure) to decide and control the distribution of the offspring points. The merit function is a standard penalty-based function that has been already proposed in the context of ES~\cite{Coello_2002}.  The main advantage of the proposed approach is guaranteeing a form of global convergence. To the best of our knowledge, this paper presents the first globally convergent framework that handles relaxable and unrelaxable constraints in the context of ES's.

The proposed convergence theory generalizes the ES framework in~\cite{YDiouane_SGratton_LNVicente_2015_b} by including relaxable constraints, all in the spirit of the proposed merit function for directional direct search methods~\cite{SGratton_LNVicente_2014}. The contributions of this paper are the following. We propose an adaptation of the merit function approach algorithm to the ES setting, a detailed convergence theory of the proposed approach is given. We provide also a practical implementation on some known global optimization problems as well as  some tests on a multidisciplinary design optimization (MDO) problem. The performance of our proposed solver are compared to the progressive barrier approach implemented in mesh adaptive direct search (MADS) solver~\cite{Audet_Dennis_2009}.

The paper is organized as follows. Section~\ref{sec:unrelax} reminds the class of ES algorithms proposed in~\cite{YDiouane_SGratton_LNVicente_2015_b} to handle unrelaxable constraints. 
The proposed merit function approach is then given in
Section~\ref{sec:merit:algo} with a detailed description of the changes introduced in a class of 
ES algorithms in order to handle general constraints. The convergence results of the adapted 
approach are then detailed in Section~\ref{sec:global_convergence}. In Section~\ref{sec:numerical_illustration}, we test the proposed algorithm on well-known constrained optimization test problems and an MDO problem.  Finally, we conclude the paper in Section 
~\ref{sec:conclusion} with some conclusions and prospects of  future work.

\section{A globally convergent ES for unrelaxable constraints}
\label{sec:unrelax}
This paper focuses on a class of ES's, denoted by $(\mu/\mu_W,\lambda)$--ES, which evolves a single candidate solution. 
In fact, at the $k$-th iteration, a new population $y_{k+1}^1,\ldots,
y_{k+1}^{\lambda}$ (called offspring) is generated around a weighted mean~$x_k$ of the previous parents (candidate solution).
The symbol ``$/\mu_W$'' in $(\mu/\mu_W,\lambda)$--ES specifies that~$\mu$ parents
are ``recombined'' into a weighted mean. The parents are selected as the $\mu$~best offspring 
of the previous iteration in terms of the objective function value.
The mutation operator of the new offspring points is done
by $y_{k+1}^i = x_k +\sigma_{k}^{\es} d_k^i$, $i=1,\ldots,\lambda$, where
$d_k^i$ is drawn from a certain distribution $\mathcal{C}_k$ and $\sigma_{k}^{\es}$
is a chosen step size. 
The weights used to compute the means belong to the simplex set
$S = \{ (\omega^1,\ldots,\omega^{\mu}) \in \mathbb{R}^{\mu}: \sum_{i=1}^{\mu} w^i = 1,
w^i \geq 0, i=1,\ldots,\mu \}$. The $(\mu/\mu_W,\lambda)$--ES adapts the sampling distribution to the landscape of the objective function. 
An adaptation mechanism for the step size parameter is also possible. The latter one increases or 
decreases depending on the landscape of the objective function. One relevant instance of
such an ES is covariance matrix adaptation ES (CMA-ES) \cite{NHansen_AOstermeier_AGawelczyk_1995}.

In~\cite{YDiouane_SGratton_LNVicente_2015_a,YDiouane_SGratton_LNVicente_2015_b}, the authors proposed a framework for making a class of ES's enjoying some global convergence properties while solving optimization problems possibly with unrelaxable constraints. In fact, in~\cite{YDiouane_SGratton_LNVicente_2015_a}, by imposing a sufficient decreasing condition on the objective function value, the proposed algorithm was
monitoring the step size~$\sigma_k$ to ensure its convergence to zero (which leads then to the existence of a stationary point). The imposed sufficient decreasing condition applied
directly to the weighted mean~$x_{k+1}^{\trial}$ of the new parents. By sufficient decreasing condition we mean $f(x_{k+1}^{\trial}) \leq f(x_k) - \rho(\sigma_k)$, where $\rho(\cdot)$ is
a forcing function~\cite{TGKolda_RMLewis_VTorczon_2003}, i.e.,
a positive, nondecreasing function satisfying
$\rho(\sigma)/\sigma \rightarrow 0$ when $\sigma \rightarrow 0$.  To handle unrelaxable constraints~\cite{YDiouane_SGratton_LNVicente_2015_b}, one
starts with a feasible iterate $x_0$ and then prevents stepping outside the feasible region by means
of a barrier approach. In this context, the sufficient decrease condition is applied not to $f$ but to the extreme barrier function $f_{\Omega_{\nr}}$ associated to $f$ with respect to the constraints set $\Omega_{\nr}$~\cite{CAudet_JEDennis_2006} (also known as the death penalty function in the terminology of evolutionary algorithms), which is defined by:
\begin{equation} \label{extreme}
f_{\Omega_{\nr}}(x) \; = \; \left\{\begin{array}{ll}f(x) & \textrm{if } x \in {\Omega_{\nr}},\\ +\infty & \textrm{otherwise.}
\end{array}\right.
\end{equation}
We consider that ties of $+\infty$ are broken arbitrarily in the ordering of the offspring samples. The obtained globally convergent ES is  given by Algorithm~\ref{alg:ESgc}.

 \LinesNumberedHidden
\begin{algorithm}[!ht]
\SetAlgoNlRelativeSize{0}
\caption{\bf \bf A globally convergent ES for unrelaxable constraints ($\Omega= \Omega_{\nr}$)}
\label{alg:ESgc}
\SetAlgoLined
\KwData{choose positive integers $\lambda$ and $\mu$ such that $\lambda \geq \mu$.
Select an initial $x_0 \in \Omega_{\nr}$ and evaluate $f(x_0)$.
Choose initial step lengths $\sigma_0,\sigma_0^{\es} > 0$ and initial
weights $(\omega_0^1,\ldots,\omega_0^{\mu}) \in S$. Choose constants $\beta_1,\beta_2,d_{\min},d_{\max}$ such
that $0 < \beta_1 \leq \beta_2 < 1$ and $0 < d_{\min} < d_{\max}$. Select a forcing function $\rho(\cdot)$.}
\For{$k= 0,  1, \ldots$}{
\textbf{Step 1:} compute new sample points $Y_{k+1} = \{ y_{k+1}^1,\ldots,y_{k+1}^{\lambda} \}$
such that
\begin{equation*} \label{mean}
y_{k+1}^i \; = \; x_k + \sigma_k \tilde{d}_k^i, \; i=1,\ldots,\lambda,
\end{equation*}
where the directions $\tilde{d}_k^i$'s are computed from the original ES directions $d_k^i$'s
(which in turn are drawn from a chosen ES distribution $\mathcal{C}_k$ and scaled if necessary to satisfy
 $d_{\min} \leq \|d_k^i \| \leq d_{\max}$)\;
 
\textbf{Step 2:} evaluate $f_{\Omega_{\nr}}(y_{k+1}^i)$, $i=1,\ldots,\lambda$, and reorder the
offspring points in $Y_{k+1} = \{ \tilde{y}_{k+1}^1,\ldots,\tilde{y}_{k+1}^{\lambda} \}$
by increasing order: $f_{\Omega_{\nr}}(\tilde{y}_{k+1}^1) \leq \cdots \leq f_{\Omega_{\nr}}(\tilde{y}_{k+1}^{\lambda})$.

Select the new parents as the best $\mu$ offspring sample points
$\{ \tilde{y}_{k+1}^1,\ldots,\tilde{y}_{k+1}^{\mu} \}$, and compute their weighted mean
\[
x_{k+1}^{\trial} \; = \; \sum_{i=1}^{\mu} \omega_k^i \tilde{y}_{k+1}^i.
\]
Evaluate $f_{\Omega_{\nr}}(x_{k+1}^{\trial})$\;
\textbf{Step 3:} \eIf{$f_{\Omega_{\nr}}(x_{k+1}^{\trial}) \leq f(x_k) - \rho(\sigma_k)$}{consider the iteration successful, set $x_{k+1} = x_{k+1}^{\trial}$, and
$\sigma_{k+1} \geq \sigma_k$ (for example $\sigma_{k+1} = \max \{ \sigma_k, \sigma_k^{\es} \}$).\; }{consider the iteration unsuccessful, set $x_{k+1} = x_k$ and
$\sigma_{k+1} = \bar{\beta}_k\sigma_k$, with $\bar{\beta}_k \in (\beta_1,\beta_2)$\;}

\textbf{Step 4:} update the ES step length $\sigma_{k+1}^{\es}$, the distribution $\mathcal{C}_{k+1}$, and the weights
$(\omega_{k+1}^1,$ $\ldots,\omega_{k+1}^{\mu}) \in S$\;
}
\end{algorithm}

We note that, due to the local geometry of the unrelaxable constraints around the current point $x_k$, the directions ~$\tilde{d}_k^i$ used to compute the offspring are not necessarily randomly generated following a pure ES paradigm. In fact, two approaches were proposed in \cite{YDiouane_SGratton_LNVicente_2015_b}, the first one was based on the use of the extreme barrier function and also the inclusion of positive generators of the tangent cone of the constraints. In this case, whenever the current iterate $x_k$ is getting close to the boundary of the feasible region, the set of directions $\{\tilde{d}_k^i \}$ will include (in addition to the ES randomly generated directions $d_k^i$) positive generators
of the tangent cone. The second proposed approach for generating the set of directions  $\{\tilde{d}_k^i \}$ is based on projecting onto the feasible  domain all the generated sampled points $x_k + \sigma_k d_k^i$,
and then taking instead $\Phi_{\Omega_{\nr}} (x_k + \sigma_k d_k^i)$ where $\Phi_{\Omega_{\nr}}$ is a given projection operator on the constraints set $\Omega_{\nr}$. 
This procedure is the same as considering
$ \tilde{d}_k^i \; = \; \frac{\Phi_{\Omega_{\nr}} (x_k + \sigma_k d_k^i) - x_k}{\sigma_k} $
in the framework of Algorithm~\ref{alg:ESgc}. For typical choices of the projection $\Phi_{\Omega_{\nr}}$, one can use
the $\ell_2$-projection in the case of bound constraints (as it is trivial to evaluate), in the case of linear constraints  one may  use
the $\ell_1$-projection (as it reduces to the solution of an LP problem). For both approaches, we note that Steps~2 and 3 of Algorithm~\ref{alg:ESgc} make use of the extreme barrier function~(\ref{extreme}).

Due to the sufficient decrease condition, one can guarantee
that a subsequence of step sizes will converge to zero. From this property
and the fact that the step size is significantly reduced (at least by $\beta_2$) in
unsuccessful iterations, one proves
that there exists a subsequence~$K$ of unsuccessful iterates driving the step size to zero \cite[Lemma 2.1]{YDiouane_SGratton_LNVicente_2015_b}. 
The global convergence is then achieved by establishing that some type of
directional derivatives are nonnegative at limit points of refining
subsequences along certain limit directions (see \cite[Theorem 2.1]{YDiouane_SGratton_LNVicente_2015_b}).
\section{A globally convergent ES for general constraints}
\label{sec:merit:algo}

The challenge of this paper consists in changing Algorithm~\ref{alg:ESgc}, in a minimal way,
to a globally convergent framework that takes into account both relaxable constraints and unrelaxable constraints. For this, we define a merit function as follows: 
\begin{equation} \label{merit:func}
M(x) \; = \; \left\{\begin{array}{ll}f(x) + \bar{\delta}  g(x)& \textrm{if } x \in \Omega_{\nr},\\ +\infty & \textrm{otherwise.}
\end{array}\right.
\end{equation}
where $ \bar{\delta} > 0$ is a given positive constant and $g$ defines a constraint violation function with respect to relaxable constraints. The $\ell_1$-norm is commonly used to define the constraint violation function, i.e.,  $g(x) =  \sum_{i=1}^{r} \max(c_i(x),0)$. Other choices for $g$ exist, for instance, using the $\ell_2$-norm i.e., $g(x) = \sum^r_{i=1} \max(c_i(x),0)^2$.
The merit function will be used to evaluate a trial step and hence decide whether such step will be accepted or not.
The extension of the globally convergent ES (given by Algorithm~\ref{alg:ESgc}) to a general constrained setting can be seen as a combination of two approaches, a feasible one
where either the extreme barrier or a projection operator will be used to handle the unrelaxable constraints, and a merit function approach (possibly with a restoration procedure) to handle relaxable constraints. 

The description of the proposed framework is as follows. For a given iteration $k$, a trial mean parent~$x_{k+1}^{\trial}$ is computed as the weighted mean of the $\mu$ best points in terms 
of the merit function value. The current trial mean parent will be considered as a ``\textit{\textbf{Successful point}}'' if one of the two following situations occur. 
The first scenario happens when one is sufficiently away from the feasible region (i.e., $g(x_k) > C \rho(\sigma_k)$ for some constant $C > 1$) and $x_{k+1}^{\trial}$ 
sufficiently decreases the constraint violation function $g$ is observed (i.e., $g_{\Omega_{\nr}}(x_{k+1}^{\trial}) < g(x_k) - \rho(\sigma_k)$, where $g_{\Omega_{\nr}}$ denotes the extreme barrier function associated to $g$ with respect to $\Omega_{\nr}$). The second situation occurs when the merit function is sufficiently decreased (i.e., $M(x_{k+1}^{\trial})  <  M(x_k) - \rho(\sigma_k)$). 

%
%

Before checking whether the trial point is successful or not, the algorithm will try first to restore 
the feasibility or at least decrease the constraints violation if needed. The restoration process will be activated if the current mean parent $x_k$ is far away from the feasible region and the trial point $x_{k+1}^{\trial}$ sufficiently decreases the constraint violation function $g$ but not the merit function. 
More specifically, a ``\textit{\textbf{Restoration identifier}}'' will be activated if one has 
\[
  g_{\Omega_{\nr}}(x_{k+1}^{\trial})  \; < \; g(x_k) - \rho(\sigma_k) ~~~\mbox{and} ~~~~ g(x_k) > C \rho(\sigma_k)
\]
and
\[
 M(x_{k+1}^{\trial}) \; \ge \; M(x_k).
\]

%

The restoration algorithm will be left as far as progress on the reduction of the constraint violation can not be achieved all without any considerable increase in $f$. The complete description of the restoration procedure is given in Algorithm \ref{alg:M-ES-Resotoration}. 

As a result, the main iteration of the proposed merit function approach can be devided to two steps: restoration and minimization. In the restoration step the
aim is to decrease infeasibility (by minimizing essentially the function $g_{\Omega_{\nr}}$) while in the minimization step the objective function $f$ is improved over a relaxed set of constraints by using the merit function $M$. The final obtained approach  is described is given in Algorithm \ref{alg:M-ES-Main}. 

 \LinesNumberedHidden
\begin{algorithm}[!ht]
\SetAlgoNlRelativeSize{0}
\caption{\bf \bf A globally convergent ES for general constraints (Main)}
\label{alg:M-ES-Main}
\SetAlgoLined
\KwData{choose positive integers $\lambda$ and $\mu$ such that $\lambda \geq \mu$.
Select an initial $x_0 \in \Omega_{\nr}$ and evaluate $f(x_0)$.
Choose initial step lengths $\sigma_0,\sigma_0^{\es} > 0$ and initial
weights $(\omega_0^1,\ldots,\omega_0^{\mu}) \in S$. Choose constants $\beta_1,\beta_2,d_{\min},d_{\max}$ such
that $0 < \beta_1 \leq \beta_2 < 1$ and $0 < d_{\min} < d_{\max}$. Select a forcing function $\rho(\cdot)$.}
\For{$k= 0,  1, \ldots$}{
\textbf{Step 1:} compute new sample points $Y_{k+1} = \{ y_{k+1}^1,\ldots,y_{k+1}^{\lambda} \}$
such that
\begin{equation*} \label{mean}
y_{k+1}^i \; = \; x_k + \sigma_k \tilde{d}_k^i, \; i=1,\ldots,\lambda,
\end{equation*}
where the directions $\tilde{d}_k^i$'s are computed from the original ES directions $d_k^i$'s
(which in turn are drawn from a chosen ES distribution $\mathcal{C}_k$ and scaled if necessary to satisfy
 $d_{\min} \leq \|d_k^i \| \leq d_{\max}$).\;
 
\textbf{Step 2:} evaluate $M(y_{k+1}^i)$, $i=1,\ldots,\lambda$, and reorder the
offspring points in $Y_{k+1} = \{ \tilde{y}_{k+1}^1,\ldots,\tilde{y}_{k+1}^{\lambda} \}$
by increasing order: $M(\tilde{y}_{k+1}^1) \leq \cdots \leq M(\tilde{y}_{k+1}^{\lambda})$.

Select the new parents as the best $\mu$ offspring sample points
$\{ \tilde{y}_{k+1}^1,\ldots,\tilde{y}_{k+1}^{\mu} \}$, and compute their weighted mean
\[
x_{k+1}^{\trial} \; = \; \sum_{i=1}^{\mu} \omega_k^i \tilde{y}_{k+1}^i;
\]
\textbf{Step 3:} \eIf{$x_{k+1}^{\trial} \notin \Omega_{\nr} $}{ the iteration is declared unsuccessful\; }{ \eIf{$x_{k+1}^{\trial}$ is a ``\textbf{Restoration identifier}'' }{  enter Restoration (with $k_r =k $)\; }{ \eIf{$x_{k+1}^{\trial}$ is a ``\textbf{Successful point}''}{ declare the iteration successful, set $x_{k+1} = x_{k+1}^{\trial}$, and
$\sigma_{k+1} \geq \sigma_k$ (for example $\sigma_{k+1} = \max \{ \sigma_k, \sigma_k^{\es} \}$)\; }{ the iteration is declared unsuccessful\;}}}
\If{the iteration is declared unsuccessful}{ set $x_{k+1} = x_k$ and
$\sigma_{k+1} = \beta_k \sigma_k$, with $\beta_k \in (\beta_1,\beta_2)$\; }

\textbf{Step 4:} update the ES step length $\sigma_{k+1}^{\es}$, the distribution $\mathcal{C}_{k+1}$, and the weights
$(\omega_{k+1}^1,$ $\ldots,\omega_{k+1}^{\mu}) \in S$\;
}
\end{algorithm}


 \LinesNumberedHidden
\begin{algorithm}[!ht]
\SetAlgoNlRelativeSize{0}
\caption{\bf \bf A globally convergent ES for general constraints (Restoration)}
\label{alg:M-ES-Resotoration}
\SetAlgoLined
\KwData{Start from $x_{k_r} \in \Omega_{\nr}$ given from the Main algorithm and consider the same parameter as in there.}
\For{$k=k_r, k_r +1 , k_r +2, \ldots$}{
\textbf{Step 1:} compute new sample points $Y_{k+1} = \{ y_{k+1}^1,\ldots,y_{k+1}^{\lambda} \}$
such that
\begin{equation*} \label{mean}
y_{k+1}^i \; = \; x_k + \sigma_k \tilde{d}_k^i, \; i=1,\ldots,\lambda,
\end{equation*}
where the directions $\tilde{d}_k^i$'s are computed from the original ES directions $d_k^i$'s
(which in turn are drawn from a chosen ES distribution $\mathcal{C}_k$ and scaled if necessary to satisfy
 $d_{\min} \leq \|d_k^i \| \leq d_{\max}$)\;
 
\textbf{Step 2:}  evaluate $g_{\Omega_{\nr}}(y_{k+1}^i)$, $i=1,\ldots,\lambda$, and reorder the
offspring points in $Y_{k+1} = \{ \tilde{y}_{k+1}^1,\ldots,\tilde{y}_{k+1}^{\lambda} \}$
by increasing order: $g_{\Omega_{\nr}}(\tilde{y}_{k+1}^1) \leq \cdots \leq g_{\Omega_{\nr}}(\tilde{y}_{k+1}^{\lambda})$.

Select the new parents as the best $\mu$ offspring sample points
$\{ \tilde{y}_{k+1}^1,\ldots,\tilde{y}_{k+1}^{\mu} \}$, and compute their weighted mean
\[
x_{k+1}^{\trial} \; = \; \sum_{i=1}^{\mu} \omega_k^i \tilde{y}_{k+1}^i;
\]
\textbf{Step 3:} \eIf{$x_{k+1}^{\trial} \notin \Omega_{\nr} $}{ the iteration is declared unsuccessful\; }{ \eIf{ $g(x_{k+1}^{\trial}) \; < \; g(x_k) - \rho(\sigma_k) ~~~\mbox{and} ~~~~ g(x_k) > C \rho(\sigma_k)$ }{  the iteration is declared successful, set $x_{k+1} = x_{k+1}^{\trial}$, and
$\sigma_{k+1} \geq \sigma_k$ (for example $\sigma_{k+1} = \max \{ \sigma_k, \sigma_k^{\es} \}$)\; }{ the iteration is declared unsuccessful\;}}
\If{the iteration is declared unsuccessful}{ \eIf{$  M(x_{k+1}^{\trial}) \; < \; M(x_k) $}{leave Restoration and return to the Main algorithm (starting at a new $(k+1)$-th 
iteration using $x_{k+1}$ and $\sigma_{k+1}$)\;}{ set $x_{k+1} = x_k$ and
$\sigma_{k+1} = \beta_k \sigma_k$, with $\beta_k \in (\beta_1,\beta_2)$\; } }

\textbf{Step 4:} update the ES step length $\sigma_{k+1}^{\es}$, the distribution $\mathcal{C}_{k+1}$, and the weights
$(\omega_{k+1}^1,$ $\ldots,\omega_{k+1}^{\mu}) \in S$\;
}
\end{algorithm}

For both algorithms (main and restoration), our global convergence analysis will be done independently of the choice of the distribution $\mathcal{C}_k$, the weights $(\omega_{k}^1,$ $\ldots,\omega_{k}^{\mu}) \in S$, and the step size $\sigma_{k}^{\es}$ . Therefore and similarly to Algorithm \ref{alg:ESgc}, the update of the ES parameters is left unspecified. Note that one also imposes bounds on all directions $d_k^i $ used by the algorithm. This modification is, however, very
mild since the lower bound $d_{\min}$ can be chosen very close to zero and the upper bound $d_{\max}$ set to a very large number.  The construction of the set of directions $\{\tilde{d}_k^i \}$ can be done with respect to the local geometry of the unrelaxable constraints as proposed in \cite{YDiouane_SGratton_LNVicente_2015_b}.

\section{Global convergence}
\label{sec:global_convergence}
The convergence results presented in this section are in the vein of those first established for the merit function approach for direct search methods~\cite{SGratton_LNVicente_2014}. For the convergence analysis, we will consider a sequence of iterations generated by Algorithm~\ref{alg:M-ES-Main} without any stopping criterion. The analysis is organized depending on the number of times where restoration is
entered. To keep the presentation of the paper simpler, only the case where restoration is entered for finitely times will be treated in our convergence analysis of this section. For completeness, the analysis of the two other cases (namely when (a) an infinite run of consecutive steps inside Restoration or (b) one enters the restoration an infinite number of times) is given in Appendix \ref{sec:Appendix}. The analysis of both cases shows that such behaviors would lead to feasibility and optimality results similar to the case where the restoration is entered finitely times.



When the restoration is entered for finitely times, one can guarantee
that a subsequence of the step sizes $\{\sigma_k\}$ will converge to zero. In fact, due to the sufficient decrease condition imposed on the merit function along the iterates (or in the constraints
violation function if the iterates are sufficiently away from the feasible region) and the control on the step size (reduced at least by $\beta_2$ for
unsuccessful iterations), one can ensure the existence of a subsequence~$K$ of unsuccessful iterates driving the step size to zero.

\begin{lemma} \label{liminf}
Let $f$ be bounded below and assuming that the restoration is not entered after a certain order.
Then, \[
    \liminf_{k \rightarrow +\infty} \sigma_k = 0.   
      \]
\end{lemma}
\begin{proof}
Suppose that there exists a $\bar{k} > 0 $  and $\sigma > 0$ such that $\sigma_k > \sigma$ 
and $k \ge \bar{k}$ is a given iteration of Algorithm~\ref{alg:M-ES-Main}. If there is an infinite sequence $J_1$ of successful iterations after $\bar{k}$, 
this leads to a contradiction with the fact that $g$ and $f$ are bounded below.

In fact, since $\rho$ is a nondecreasing positive function, one has
$\rho(\sigma_k) \geq \rho(\sigma) > 0$. Hence,
if $g(x_{k+1}) < g(x_k) - \rho( \sigma_k )$ and $g(x_{k}) > C \rho(\sigma_k)$ for all~$k \in J_1$, then
\[
 g(x_{k+1}) < g(x_k) - \rho(\sigma),
\]
which obviously contradicts the boundness below of $g$ by $0$. 
Thus there must exists an infinite subsequence $J_2 \subseteq J_1 $ of iterates for which 
$M(x_{k+1}) \; < \; M(x_k) - \rho(\sigma_k)$. Hence,
\[
 M(x_{k+1}) \; < \; M(x_k) - \rho(\sigma)  \;  ~~~~\mbox{for all}~k \in J_2.
\]
Thus $ M(x_k)$ tends to $- \infty$ which is a contradiction, since both $f$ and $g$ are bounded below.

The proof is thus completed if there is an infinite number of 
successful iterations. However, if no more successful iterations 
occur after a certain order, then this also leads to a contradiction. 
The conclusion is that one must have a subsequence of iterations 
driving $\sigma_k$ to zero.
\end{proof}

\begin{theorem} \label{liminf2}
Let $f$ be bounded below and assuming that the restoration is not entered after a certain order.

There exists a subsequence $K$ of unsuccessful iterates 
for which $\lim_{k\in K} \sigma_k = 0$. Moreover, if the sequence $\{x_k\}$ is bounded, there 
exists an $x_*$ and a refining subsequence $K'$ such that $\lim_{k\in K} x_k = x_*$.
\end{theorem}
\begin{proof}
From Lemma~\ref{liminf}, there must exist an infinite 
subsequence $K$ of unsuccessful iterates for which 
$\sigma_{k+1}$ goes to zero. In a such case we have 
$\sigma_{k} = (1/\beta_k) \sigma_{k+1} $, $\beta_k \in (\beta_1,\beta_2)$, 
and $\beta_1>0$, and thus $\sigma_k \rightarrow 0$, for $k \in K$, too.

The second part of the theorem is proved by 
extracting a convergent subsequence $K' \subset K$
 for which $x_k$ converges to $x_*$.
\end{proof}

The global convergence will be achieved by establishing that some type of
directional derivatives are nonnegative at limit points of refining
subsequences along certain limit directions (known as refining directions). By refining subsequence~\cite{CAudet_JEDennis_2006}, we mean a subsequence
of unsuccessful iterates in the Main algorithm (see Algorithm~\ref{alg:M-ES-Main}) for which the step-size parameter converges to zero. 

When $h$ is Lipschitz continuous near $x_* \in \Omega_{\nr}$, one can make use of the Clarke-Jahn generalized derivative
along a direction~$d$
\[
h^\circ(x_*;d) \; = \;
\limsup_{\begin{array}{c}x \rightarrow x_*, x\in \Omega_{\nr}\\t\downarrow0,x+td\in\Omega_{\nr}\end{array}}\frac{h(x+td)-h(x)}{t}.
\]
(Such a derivative is essentially the Clarke generalized directional derivative~\cite{FHClarke_1990},
adapted by Jahn~\cite{JJahn_1996} to the presence of constraints).
However, for the proper definition of $h^\circ(x_*;d)$, one needs to guarantee that
$x+td\in\Omega_{\nr}$ for $x \in \Omega_{\nr}$ arbitrarily close to~$x_*$ which is
assured if $d$ is hypertangent to $\Omega_{\nr}$ at $x_*$. In the following,
$B(x;\epsilon)$ is the closed ball formed by all points which dist no more than~$\epsilon$ to~$x$.

\begin{definition}
A vector $d\in \real^n$ is said to be a hypertangent vector to the set $\Omega_{\nr} \subseteq \real^n$
at the point $x$ in $\Omega_{\nr}$ if there exists a scalar $\epsilon>0$ such that
\begin{equation*}
y+tw\in\Omega_{\nr},\quad\forall y\in\Omega_{\nr}\cap B(x;\epsilon),\quad w\in B(d;\epsilon),\quad\text{and}\quad 0<t<\epsilon.
\end{equation*}
\end{definition}

The hypertangent cone to $ \Omega_{\nr}$ at $x$, denoted by~$T_{\Omega_{\nr}}^{\h}(x)$, is the set of all hypertangent vectors
to $\Omega_{\nr}$ at $x$.
Then, the Clarke tangent cone to $\Omega_{\nr}$ at $x$ (denoted by~$T_{\Omega_{\nr}}^{\cl}(x)$)
can be defined as the closure of the hypertangent cone~$T_{\Omega_{\nr}}^{\h}(x)$.
The Clarke tangent cone generalizes the notion of tangent cone in Nonlinear Programming~\cite{Nocedal_Wright_2006},
and the original definition $d \in T_{\Omega_{\nr}}^{\cl}(x)$ is given below.

\begin{definition}
A vector $d\in \real^n$ is said to be a Clarke tangent vector to the set $\Omega_{\nr} \subseteq \real^n$
at the point $x$ in the closure of $\Omega_{\nr}$ if for every sequence $\{y_k\}$ of elements of
$\Omega_{\nr}$ that converges to $x$ and for every sequence of positive real numbers $\{t_k\}$
converging to zero, there exists a sequence of vectors $\{w_k\}$ converging to $d$ such
that $y_k+t_kw_k\in\Omega_{\nr}$.
\end{definition}

Given a direction~$v$ in the tangent cone, possibly not in the hypertangent one, one can consider the
Clarke-Jahn generalized derivative to $\Omega_{\nr}$ at $x_*$ as the limit
\[
h^\circ(x_*;v) \; = \; \lim_{d \in T_{\Omega_{\nr}}^{\h}(x_*), d \rightarrow v} h^\circ(x_*;d)
\]
(see~\cite{CAudet_JEDennis_2006}).
A point~$x_* \in \Omega_{\nr}$ is considered Clarke stationary
if $h^\circ(x_*;d) \geq 0$, $\forall d \in T_{\Omega_{\nr}}^{\cl}(x_*)$.


It remains now to define the notion
of refining direction~\cite{CAudet_JEDennis_2006},
associated with a convergent refining subsequence~$K$,
as a limit point of $\{ a_k/\|a_k\| \}$ for all $k \in K$ sufficiently large such
that $x_k + \sigma_k a_k \in \Omega_{\nr}$, where, in the particular case of taking
the weighted mean as the object of evaluation, one has $a_k = \sum_{i=1}^{\mu} \omega_k^i \tilde{d}_k^i$.
The following convergence result is concerning the determination of feasibility.
\begin{theorem} \label{th:1}
Let $a_k = \sum_{i=1}^{\mu} \omega_k^i d_k^i$ and assume that $f$ is bounded below. Suppose that the restoration is not entered after a certain order. 
Let $x_* \in \Omega_{\nr}$ be the limit point of a convergent subsequence of unsuccessful iterates $\{ x_k \}_K$ for which $\lim_{k \in K} \sigma_k = 0$.
Assume that~$g$ is Lipschitz continuous near~$x_*$ with constant $\nu_g > 0$.

If $d \in T_{\Omega_{\nr}}^{\h}(x_*)$ is a refining direction associated with $\{ a_k/\|a_k\| \}_K$, then either $g(x_*) = 0$ 
or  $g^\circ(x_*;d) \geq 0$.
\end{theorem}
\begin{proof}
Let $d$ be a limit point of $\{ a_k/\|a_k\| \}_K$. Then it must exist 
a subsequence $K'$ of $K$ such that
$a_k / \| a_k \| \rightarrow d$ on $K'$. 
On the other hand, we have for all $k$ that
\[
x_{k+1}^{\trial} \; = \; \sum_{i=1}^{\mu} \omega_k^i \tilde{y}_{k+1}^i \; = \; x_k + \sigma_k \sum_{i=1}^{\mu} \omega_k^i d_k^i
 \; = \; x_k + \sigma_k a_k,
\]
Since the iteration $k \in K'$ is unsuccessful, $g(x_{k+1}^{\trial}) \geq g(x_k) - \rho( \sigma_k )$ or $g(x_{k}) \leq C \rho(\sigma_k)$, 
and then either there exists an infinite number of the first inequality or the second one as follows:
\begin{enumerate}
\item  For the case where there exists a subsequence $K_1 \subseteq K'$ such that $g(x_{k}) \leq C \rho(\sigma_k)$, it is 
trivial to obtain  $g(x_*) = 0$ using both the continuity of g and the fact that $\sigma_k$ tends to zero in $K_1$.
\item For the case where there exists a subsequence $K_2 \subseteq K'$ such that the 
sequence $\{ a_k/\|a_k\| \}_{K_2}$ converges to $d \in T_{\Omega_{\nr}}^{\h}(x_*) $ 
in $K_2$ and the sequence $\{ \| a_k \| \sigma_k \}_{k \in K_2}$ goes to zero in $K_2$ 
($a_k$ is bounded above for all~$k$, and so $\sigma_k \|a_k\|$ tends to zero when $\sigma_k$ does). 
Thus one must have necessarily for $k$ sufficiently large in $K_2$, $x_k + \sigma_k a_k \in {\Omega_{\nr}}$ such that
\[
 g(x_k + \sigma_k a_k ) \geq g(x_k) - \rho( \sigma_k ).
\]
From the definition of the Clarke-Jahn generalized derivative along directions $d \in T_{\Omega_{\nr}}^{\h}(x_*)$,
\begin{eqnarray*}
g^{\circ}(x_*;d)&=&\limsup_{x\rightarrow x_*, t \downarrow 0, x+ t d \in {\Omega_{\nr}}} \frac{g(x+t d)-g(x)}{t}\\
 & \geq  & \limsup_{k\in K_2} \frac{g(x_k+\sigma_k \| a_k\| d)-g(x_k)}{\sigma_k \| a_k\|}\\
&\geq&\limsup_{k\in K_2}\frac{g(x_k+\sigma_k \|a_k\| ( a_k / \| a_k \| ) )-g(x_k)}{\sigma_k \| a_k\|} - g_k,
\end{eqnarray*}
where, 
\begin{eqnarray*}
g_k &  = & \frac{g(x_k + \sigma_k a_k)-g(x_k+\sigma_k \| a_k\| d)}{\sigma_k \| a_k\|} \\
\end{eqnarray*}
from the Lipschitz continuity of~$g$ near $x_*$
\begin{eqnarray*}
g_k &  = & \frac{g(x_k + \sigma_k a_k)-g(x_k+\sigma_k \| a_k\| d)}{\sigma_k \| a_k\|}  \\ 
    & \leq & \nu_g \left\| \dfrac{a_k}{\|a_k\|} - d \right\|   \\
\end{eqnarray*}
tends to zero on $K_2$. Finally,
\begin{eqnarray*}
g^{\circ}(x_*;d)& \geq &\limsup_{k\in K_2}\frac{g(x_k+\sigma_k a_k)-g(x_k)+\rho(\sigma_k)}{\sigma_k \|a_k\|}-
\frac{\rho(\sigma_k)}{\sigma_k \|a_k\|} - g_k\\
&=&\limsup_{k\in K_2}\frac{g(x_k+\sigma_k a_k) - g(x_k)+\rho(\sigma_k)}{\sigma_k \|a_k\|}.
\end{eqnarray*}
One then obtains $g^{\circ}(x_*;d) \geq 0$.
\end{enumerate}
\end{proof}

Moreover, assuming that the set of the refining directions $d \in T_{\Omega_{\nr}}^{\h}(x_*)$, associated with $\{ a_k/\|a_k\| \}_K$, is dense in the unite sphere. One can show that the limit point $x_*$ is Clarke stationary for the flowing optimization problem, known as the constraint violation problem:
\begin{eqnarray}
\label{const_viola}
 \min & g(x)   \\
  s.t. &  x \in \Omega_{\nr}. \nonumber
\end{eqnarray}
\begin{theorem} \label{th:2}
Let $a_k = \sum_{i=1}^{\mu} \omega_k^i d_k^i$ and assume that $f$ is bounded below. Suppose that the restoration is not entered after a certain order.  Assume that the directions~$\tilde{d}_k^i$'s
and the weights $\omega_k^i$'s
are such that (i) $\sigma_k \|a_k\|$ tends to zero when $\sigma_k$ does, and (ii)
$\rho(\sigma_k)/(\sigma_k \| a_k\| )$ also tends to zero.

Let $x_* \in \Omega_{\nr}$ be the limit point of a convergent subsequence of
unsuccessful iterates $\{ x_k \}_K$ for which $\lim_{k \in K} \sigma_k = 0$
and that~$T_\Omega^{\cl}(x_*) \neq \emptyset$. Assume that~$g$ is Lipschitz continuous near~$x_*$ with constant $\nu > 0$

Then either (a) $g(x_*) = 0$ (implying $x_* \in \Omega_{\re}$ and thus $x_* \in \Omega $) or (b) if the set of refining directions  $d \in T_{\Omega_{\nr}}^{\cl}(x_*)$ associated with $\{ a_k/\|a_k\| \}_{K'}$ (where $K'$ is a subsequence of K for which $g(x_k + \sigma_k a_k) \ge g(x_k) - \rho(\sigma_k)$) is dense in $T_{\Omega_{\nr}}^{\cl}(x_*) \cap \{d \in \real^n: \|d\|=1\}$, then $g^\circ(x_*;v) \geq 0$ for all $v\in T_{\Omega_{\nr}}^{\cl}(x_*) $ and  $x_*$ is a Clarke stationary point of the constraint violation problem (\ref{const_viola}).
\end{theorem}
\begin{proof}
See the proof of \cite[Theorem 4.2]{SGratton_LNVicente_2014}.
%
\end{proof}

We now move to an intermediate optimality result.  As in \cite{SGratton_LNVicente_2014}, we will not use $x_* \in \Omega_{\re} $ 
explicitly in the proof but only  $g^\circ(x_*;d) \leq 0$. The latter inequality describes the cone of 
first order linearized directions under feasibility assumption $x_* \in \Omega_{\re}$.
\begin{theorem} \label{th:2}
Let $a_k = \sum_{i=1}^{\mu} \omega_k^i d_k^i$ and assume that $f$ is bounded below.  Suppose that the restoration is not entered after a certain order.  

Let $x_* \in \Omega_{\nr}$ be the limit point of a convergent subsequence of unsuccessful iterates $\{ x_k \}_K$ for which $\lim_{k \in K} \sigma_k = 0$.
Assume that~$g$ and $f$ are Lipschitz continuous near~$x_*$.

If $d \in T_{\Omega_{\nr}}^{\h}(x_*)$ is a refining direction associated with $\{ a_k/\|a_k\| \}_K$ such that $g^\circ(x_*;d) \leq 0$. Then  $f^\circ(x_*;d) \geq 0$.

\end{theorem}
\begin{proof}
By assumption there exists a subsequence $K' \subseteq K$ such that the 
sequence $\{a_k/ \| a_k \| \}_{K'}$ converges to $d \in T_{\Omega_{\nr}}^{\h}(x_*) $ 
in $K'$ and the sequence $\{ \| a_k \| \sigma_k \}_{K'}$ goes to zero in $K'$, 
Thus one must have necessarily for $k$ sufficiently large in $K'$, $ x_{k+1}^{\trial} = x_k + \sigma_k a_k \in {\Omega_{\nr}}$. 

Since the iteration $k \in K'$ is unsuccessful, one has $M(x_{k+1}^{\trial}) \; \geq \; M(x_k) - \rho(\sigma_k)$, and thus
\begin{eqnarray}
\label{equ:2}
\frac{f(x_k + \sigma_k a_k)- f(x_k)}{ \| a_k \| \sigma_k} &  \geq &  - \bar \delta\frac{g(x_k + \sigma_k a_k)- g(x_k)}{ \| a_k \| \sigma_k}  - \frac{\rho(\sigma_k)}{\sigma_k \| a_k\|} 
\end{eqnarray}
On the other hand,
\begin{eqnarray*}
f^{\circ}(x_*;d)&=&\limsup_{x\rightarrow x_*, t \downarrow 0, x+ t d \in \Omega} \frac{f(x+t d)-f(x)}{t}\\
 & \geq  & \limsup_{k\in K'} \frac{f(x_k+\sigma_k \| a_k\| d)-f(x_k)}{\sigma_k \| a_k\|}\\
&\geq&\limsup_{k\in K'}\frac{f(x_k+\sigma_k \|a_k\| ( a_k / \| a_k \| ) )-f(x_k)}{\sigma_k \| a_k\|} - f_k,
\end{eqnarray*}
where, 
\begin{eqnarray*}
f_k &  = & \frac{f(x_k + \sigma_k a_k)-f(x_k+\sigma_k \| a_k\| d)}{\sigma_k \| a_k\|}, \\
\end{eqnarray*}
which then implies from (\ref{equ:2})
\begin{eqnarray*}
f^{\circ}(x_*;d) &\geq&\limsup_{k\in K'}\frac{f(x_k+\sigma_k \|a_k\| ( a_k / \| a_k \| ) )-f(x_k)}{\sigma_k \| a_k\|} - f_k,\\
                 &\geq&\limsup_{k\in K'} - \bar \delta\frac{g(x_k + \sigma_k a_k)- g(x_k)}{ \| a_k \| \sigma_k}  - \frac{\rho(\sigma_k)}{\sigma_k \| a_k\|} - f_k\\
                 &\geq&\limsup_{k\in K'} - \bar \delta \frac{g(x_k+\sigma_k \|a_k\| d)-g(x_k)} {\sigma_k \| a_k\|} +  \bar \delta g_k - \frac{\rho(\sigma_k)}{\sigma_k \| a_k\|} - f_k,
\end{eqnarray*}
where 
\begin{eqnarray*}
g_k &  = & \frac{g(x_k + \sigma_k a_k)-g(x_k+\sigma_k \| a_k\| d)}{\sigma_k \| a_k\|}. \\
\end{eqnarray*}

From the assumption $g^{\circ}(x_*;d) \leq 0 $, one has 
\[
\limsup_{k\in K'} \frac{g(x_k+\sigma_k \|a_k\| d)-g(x_k)} {\sigma_k \| a_k\|} \leq \limsup_{x\rightarrow x_*, t \downarrow 0, x+ t d \in {\Omega_{\nr}}} \frac{g(x+t d)-g(x)}{t} \leq 0,
\]

one obtains then 
\begin{eqnarray}
 \label{equ:3}
    f^{\circ}(x_*;d)   &\geq& \limsup_{k\in K'}   \bar \delta g_k - \frac{\rho(\sigma_k)}{\sigma_k \| a_k\|} - f_k.
\end{eqnarray}

The Lipschitz continuity of both $g$ and $f$ near $x_*$ guaranties that the quantities $f_k$ and $g_k$ tend to zero in $K'$. Thus, the proof is completed since the right-hand-side of (\ref{equ:3}) tends to zero in $K'$.
\end{proof}
\begin{theorem} \label{th:3}
Assuming that $f$ is bounded below and that Restoration is not entered after a certain order.

Let $x_* \in \Omega_{\nr}$ be the limit point of a convergent subsequence of unsuccessful iterates $\{ x_k \}_{k \in K}$ for which $\lim_{k \in K} \sigma_k = 0$.
Assume that~$g$ and $f$ are Lipschitz continuous near~$x_*$.

Assume that the set 
\begin{eqnarray}
\label{T_x}
 T(x_*)= T^{\h}_{\Omega_{\nr}}(x_*) \cap \{ d \in \real^n : \| d \| =1,   g^{\circ}(x_*,d) \leq 0 \}
\end{eqnarray}
has a non-empty interior.

Let the set of refining directions be dense in $T(x_*)$. Then $f^\circ(x_*,v) \geq 0$ for 
all $v \in T^{\cl}_{\Omega_{\nr}}(x_*)$ such that $g^\circ(x_*,v) \leq 0$, and $x_*$ is a Clarke stationary point of the problem (\ref{c:Prob1}).
\end{theorem}
\begin{proof}
See the proof of \cite[Theorem 4.4]{SGratton_LNVicente_2014}.
\end{proof}
\section{Numerical experiments}
\label{sec:numerical_illustration}

To quantify the efficiency of the proposed merit approach, we compare 
our solver with the direct search method MADS where the progressive barrier approach has been 
implemented~\cite{Audet_Dennis_2009} to handle both relaxable and unrelaxable constraints. The progressive barrier approach, proposed in MADS, enjoys similar convergence properties as for our algorithm, hence, a comparison between the two solvers is very meaningful. For MADS solver, we used the implementation given in the NOMAD package~\cite{nomad1,nomad2,nomad3}, version~3.6.1 (C++ version linked to Matlab via a mex interface),
where we enabled the option {\tt DISABLE MODELS}, meaning that no modeling is used in MADS. 
The models are disabled since our solvers are not using any modeling to speed up the convergence. 

The parameter choices of Algorithm~\ref{alg:M-ES-Main} and Algorithm~\ref{alg:M-ES-Resotoration} followed those
in \cite{YDiouane_SGratton_LNVicente_2015_b}.
The values of $\lambda$ and $\mu$ and of the initial weights
are those of CMA-ES for unconstrained optimization (see~\cite{NHansen_2011}):
$\lambda = 4+\mbox{floor}(3\log(n))$,
$\mu = \mbox{floor}(\lambda/2)$,
where $\mbox{floor}(\cdot)$ rounds to the nearest integer,
and
$\omega_0^i = a_i/(a_1+\cdots + a_\mu)$,
$a_i = \log(\lambda/2+1/2)-\log(i)$, $i=1,\ldots,\mu$.
The choices of the distribution~$\mathcal{C}_k$ and of the update
of $\sigma_k^{\es}$ also followed CMA-ES
for unconstrained optimization (see~\cite{NHansen_2011}).
The forcing function selected
was $\rho(\sigma)=10^{-4} \sigma^2$. To reduce the step length in
unsuccessful iterations we used $\sigma_{k+1} = 0.9 \sigma_k$
which corresponds to setting $\beta_1 = \beta_2 = 0.9$.
In successful iterations we set $\sigma_{k+1} = \max \{ \sigma_k, \sigma_k^{\cmaes} \}$
(with~$\sigma_k^{\cmaes}$ the CMA step size used in ES).
The directions $d_k^i$, $i=1,\ldots,\lambda$, were
scaled if necessary to obey the safeguards $d_{\min} \leq \|d_k^i \| \leq d_{\max}$,
with $d_{\min}= 10^{-10}$ and $d_{\max}= 10^{10}$. For the update of the penalty parameter we picked $\bar \delta= \max\{10,g(x_0) \}$ and $C=100$. The measure for constraint violation was using the $\ell_1$-norm penalty.

 The initial step size is estimated
using only the bound constraints:
If there is a pair of finite lower and upper bounds for a variable, then
$\sigma_0$ is set to the half of the minimum of such distances, otherwise $\sigma_0=1$.

In our test results, we consider that all the constraints are 
relaxable except the bounds. In this case, the merit function (MF) and the 
progressive approaches (PB) are respectively enabled, the related solvers 
will be called ES-MF and MADS-PB respectively. 

\subsection{Results on known test problems}
Our test set is the one used in~\cite{Hedar_2004,Hock_Schittkowski_1981,Koziel_1999,Michalewicz_1996}
and comprises~13 well-known test problems {\tt G1}--{\tt G13} (see Table~\ref{table:test:problems}).
These test problems, coded in Matlab,  exhibit a diversity of features and the kind of difficulties that appear in
constrained global optimization. In addition to such problems, we added three other
engineering optimization problems~\cite{Hedar_2004,Coello_Montes_2002}: {\tt PVD} the pressure vessel design problem, {\tt TCS} the tension-compression string problem, and {\tt WBD} the welded beam
design problem.
Problems {\tt G2}, {\tt G3},
and {\tt G8} are maximization problems and were converted to minimization.
Problems {\tt G3}, {\tt G5}, {\tt G11}, and {\tt WBD} contain equality constraints.
When a constraint is of the form $c^e_i(x) = 0$, we use the following relaxed inequality
constraint instead $c_i(x) = |c^e_i(x)| -10^{-4} \le  0$. 

The starting point $x_0$ is chosen to be the same for all solvers and set to $(LB + UB)/2$ when 
the bounds $LB$ and $UB$ are given, otherwise it is generated randomly in the search space.
Two  different maximal budgets are considered for our experiments; firstly, we use a small one (i.e., $1000$ objective function evaluations) to analyse the performance of the algorithms during the early stages of the optimization; secondly, a large maximal budget (i.e., $20000$ objective function evaluations) is used to allow the analysis of the asymptotic behavior of the tested solvers.
We note that MADS-PB is a deterministic solver while ES-MF is stochastic, thus different runs of ES-MF will be used. We describe our finding based upon average results over 10 runs (as different runs of ES-MF yielded close results).

\begin{table}[h]
\centering
\begin{tabular}{lcccccccccccccccc}
\hline
        Name   &{\tt G1}&{\tt G2}&{\tt G3}&{\tt G4}&{\tt G5}&{\tt G6}&{\tt G7}&{\tt G8}&{\tt G9}&{\tt G10}&{\tt G11} &{\tt G12}&{\tt G13}&{\tt PVD}&{\tt TCS}&{\tt WBD}\\
        $n$     & 13 & 20 & 20 & 5 & 4 & 2 & 10 & 2 & 7   & 8   &  2  & 3 & 5 & 4 & 3 & 4 \\
        $m$     & 9 & 2 & 1 & 6 & 5 & 2 & 8 & 2 & 4   & 6   &  1 & 1 & 3& 3 & 4 & 6  \\ \hline
\end{tabular}
\caption{Some of the features of the non-linear constrained optimization problems:
the dimension~$n$ and the number of the constraints $m$ (in addition to the bounds). \label{table:test:problems}}
\end{table}



Tables~\ref{table:merit:1000} and \ref{table:merit:20000}, report results for 
both ES-MF and MADS-PB using a maximal budget of $1000$ and $20000$, respectively. For each problem,
we display the optimal objective value found by the solver $f(x_*)$, the associated constrained 
violation $g(x_*)$, and the number of objective function evaluations $\#f$ needed to reach $x_*$. When a solver 
returns a flag error or encounters an internal problem, we display '$-$' instead of the values of $f(x_*)$ and $g(x_*)$.  
At the solution $x_*$, one requires at least  tolerance of $10^{-5}$ on the constraints violation (i.e. $g(x_*)< 10^{-5}$) to consider $x_*$ feasible with respect to relaxable constraints. 

\begin{table}[!ht] 
\centering 
\begin{tabular}{|c|c|ccc|ccc|} 
\hline 
Name  & Best known &  \multicolumn{3}{|c|}{ES-MF} & \multicolumn{3}{|c|}{MADS-PB}\\ 
\cline{3-8} 
    &        $f_{opt}$  & $f(x_*)$ & $\#f$  & $g(x_*)$  & $f(x_*)$ & $\#f$  & $g(x_*)$ \\ 
\hline 
\hline 
 \multicolumn{8}{|c|}{Starting feasible} \\
\hline 
\hline 
{\tt G1 }   &  $-15$ & $-12.4895$ & $1000$  &  $4.3e-07$ &  $-8.99982$ &  $1000$ & $ 0$ \\ 
{\tt G2 }   &  $-0.803619$ & $-0.271234$ & $1000$  &  $ 0$ &  $-0.173894$ &  $1000$ & $ 0$ \\ 
{\tt G3 }   &  $-1$ & $-0.254056$ & $1000$  &  $2.7e-06$ &  $-0.0436105$ &  $1000$ & $ 0$ \\ 
{\tt G4 }   &  $-30665.5$ & $-30723.2$ & $1000$  &  $0.016$ &  $-30498.1$ &  $1000$ & $ 0$ \\ 
{\tt G5 }   &  $5126.5$ & $5999.48$ & $1000$  &  $7.6e-05$ &  $5976.11$ &  $973$ & $ 0$ \\ 
{\tt G6 }   &  $-6961.81$ & $-7588.41$ & $1000$  &  $0.22$ &  $-6961.26$ &  $1000$ & $ 0$ \\ 
{\tt G7 }   &  $24.3062$ & $147.259$ & $320$  &  $ 0$ &  $30.0327$ &  $1000$ & $ 0$ \\ 
{\tt G8 }   &  $-0.095825$ & $-0.095825$ & $330$  &  $ 0$ &  $-0.095825$ &  $453$ & $ 0$ \\ 
{\tt G9 }   &  $680.63$ & $691.948$ & $1000$  &  $ 0$ &  $683.871$ &  $1000$ & $ 0$ \\ 
{\tt G10}   &  $7049.33$ & $16607.4$ & $1000$  &  $ 0$ &  $7843.26$ &  $1000$ & $ 0$ \\ 
{\tt G11}   &  $0.75$ & $0.749403$ & $1000$  &  $2.5e-07$ &  $0.9998$ &  $331$ & $ 0$ \\ 
{\tt G12}   &  $-1$ & $-1$ & $161$  &  $ 0$ &  $-1$ &  $309$ & $ 0$ \\ 
{\tt G13}   &  $0.0539498$ & $1.45074$ & $1000$  &  $5.9e-07$ &  $2.78621$ &  $1000$ & $ 0$ \\ 
{\tt PVD}   &  $5868.76$ & $3.21995e+06$ & $1000$  &  $2.2$ &  $8115.01$ &  $978$ & $ 0$ \\ 
{\tt TCS}   &  $0.0126653$ & $0.0135886$ & $817$  &  $ 0$ &  $0.0126658$ &  $836$ & $ 0$ \\ 
{\tt WBD}   &  $1.725$ & $3.13076$ & $559$  &  $ 0$ &  $3.01286$ &  $1000$ & $ 0$ \\ 
\hline 
\hline 
 \multicolumn{8}{|c|}{Starting infeasible} \\ 
\hline 
\hline 
{\tt G1 }   &  $-15$ & $-11.0679$ & $1000$  &  $ 0$ &  $-8.93833$ &  $1000$ & $ 0$ \\ 
{\tt G2 }   &  $-0.803619$ & $-0.271234$ & $1000$  &  $ 0$ &  $-0.173894$ &  $1000$ & $ 0$ \\ 
{\tt G3 }   &  $-1$ & $-0.000743875$ & $1000$  &  $ 0$ &  $-1.40301e-06$ &  $1000$ & $ 0$ \\ 
{\tt G4 }   &  $-30665.5$ & $-31003.2$ & $1000$  &  $0.22$ &  $-30643.8$ &  $1000$ & $ 0$ \\ 
{\tt G5 }   &  $5126.5$ & $5603.69$ & $1000$  &  $4.4e+04$ &  $5236.08$ &  $1000$ & $0.63$ \\ 
{\tt G6 }   &  $-6961.81$ & $-3351.17$ & $1000$  &  $ 0$ &  $-6961.81$ &  $1000$ & $ 0$ \\ 
{\tt G7 }   &  $24.3062$ & $49.1948$ & $1000$  &  $ 0$ &  $83.6455$ &  $1000$ & $ 0$ \\ 
{\tt G8 }   &  $-0.095825$ & $-0.095825$ & $204$  &  $ 0$ &  $-0.095825$ &  $525$ & $ 0$ \\ 
{\tt G9 }   &  $680.63$ & $691.948$ & $1000$  &  $ 0$ &  $683.871$ &  $1000$ & $ 0$ \\ 
{\tt G10}   &  $7049.33$ & $9626.33$ & $1000$  &  $4.8$ &  $6013.14$ &  $1000$ & $0.031$ \\ 
{\tt G11}   &  $0.75$ & $0.749403$ & $1000$  &  $2.5e-07$ &  $0.9998$ &  $331$ & $ 0$ \\ 
{\tt G12}   &  $-1$ & $-1$ & $161$  &  $ 0$ &  $-1$ &  $309$ & $ 0$ \\ 
{\tt G13}   &  $0.0539498$ & $1$ & $1000$  &  $ 1$ &  $0.998918$ &  $1000$ & $ 0$ \\ 
{\tt PVD}   &  $5868.76$ & $2284.14$ & $1000$  &  $2.2$ &  $6344.92$ &  $997$ & $ 0$ \\ 
{\tt TCS}   &  $0.0126653$ & $0.000149129$ & $617$  &  $0.97$ &  $-$ &  $-$ & $ -$ \\ 
{\tt WBD}   &  $1.725$ & $3.78146$ & $1000$  &  $ 0$ &  $3.89919$ &  $1000$ & $ 0$ \\ 
\hline 
\end{tabular} 
\caption{Obtained results using a maximal budget of $1000$ (average of $10$ runs).\label{table:merit:1000}}
\end{table} 

Table~\ref{table:merit:1000} gives the obtained results for a maximal budget of $1000$ function evaluations. For both starting strategies (feasible or not) and except few problems, the tested solvers were not able to converge with the regarded budget. ES-MF is shown to have comparable performance with MADS-PB on the tested problems. In fact, with a feasible starting point,  ES-MF is performing well on the problems  {\tt G1}, {\tt G2}, {\tt G3}, {\tt G8}, {\tt G11}, {\tt G12}, and {\tt G13}. While MADS-PB is being the best on the problems  {\tt G4}, 
{\tt G5}, {\tt G6}, {\tt G7}, {\tt G9}, {\tt G10}, {\tt PVD} , {\tt TCS} and {\tt WBD}. Using an infeasible starting point, ES-MF is performing better on the problems {\tt G1}, {\tt G2}, {\tt G3},  {\tt G7}, {\tt G8}, {\tt G11}, {\tt G12}, {\tt TCS}, and {\tt WBD}.

For a large maximal number of function evaluation of $20000$ (Table~\ref{table:merit:20000}), ES-MF and MADS-PB
achieve convergence to a stationary point on more problems. We note that MADS-PB requires more function evaluations for four problems {\tt G2}, {\tt G3}, {\tt G10} and {\tt G13}.
When the starting point is feasible, enlarging the budget  allows having exact feasibility at the solution point. In this case, regarding the value of the objective function, we note that increasing the number of function evaluation does not affect the results compare to the ones obtained using $1000$ function evaluations.
When using an infeasible starting point, the advantage of ES-MF over MADS-PB is more evident compared to the small budget case. In fact, one can observe that ES-MF 
is better than MADS-PB on nine of the 
sixteen tested problems (i.e. {\tt G1}, {\tt G2}, {\tt G5}, {\tt G6}, {\tt G7}, {\tt G8},   {\tt G9},  {\tt G11},  {\tt G12}, {\tt G13} and {\tt WBD}). MADS-PB is shown to be better on the following four problems: {\tt G3}, {\tt G4}, {\tt G6} and {\tt PVD}.  Both solvers did not succeed to find a feasible solution for the problem {\tt G10} , for the {\tt TCS} problem MADS returns a flag error while ES-MF converge to an unfeasible solution.

\begin{table}[!ht] 
\centering 
\begin{tabular}{|c|c|ccc|ccc|} 
\hline 
Name  & Best known &  \multicolumn{3}{|c|}{ES-MF} & \multicolumn{3}{|c|}{MADS-PB}\\ 
\cline{3-8} 
    &        $f_{opt}$  & $f(x_*)$ & $\#f$  & $g(x_*)$  & $f(x_*)$ & $\#f$  & $g(x_*)$ \\ 
\hline 
\hline 
 \multicolumn{8}{|c|}{Starting feasible} \\
\hline 
\hline 
{\tt G1 }   &  $-15$ & $-12.9999$ & $7645$  &  $0$ &  $-9.00207$ &  $4817$ & $ 0$ \\ 
{\tt G2 }   &  $-0.803619$ & $-0.27127$ & $2107$  &  $ 0$ &  $-0.226599$ &  $20000$ & $ 0$ \\ 
{\tt G3 }   &  $-1$ & $-1.05473$ & $5227$  &  $0$ &  $-0.652199$ &  $20000$ & $ 0$ \\ 
{\tt G4 }   &  $-30665.5$ & $-31009.8$ & $3664$  &  $0.17$ &  $-30503.1$ &  $3064$ & $ 0$ \\ 
{\tt G5 }   &  $5126.5$ & $5976.79$ & $244$  &  $ 0$ &  $5976.79$ &  $1132$ & $ 0$ \\ 
{\tt G6 }   &  $-6961.81$ & $-6942.57$ & $1261$  &  $0$ &  $-6961.81$ &  $1381$ & $ 0$ \\ 
{\tt G7 }   &  $24.3062$ & $147.259$ & $320$  &  $ 0$ &  $25.5112$ &  $5160$ & $ 0$ \\ 
{\tt G8 }   &  $-0.095825$ & $-0.095825$ & $330$  &  $ 0$ &  $-0.095825$ &  $453$ & $ 0$ \\ 
{\tt G9 }   &  $680.63$ & $680.63$ & $5071$  &  $0$ &  $680.799$ &  $3568$ & $ 0$ \\ 
{\tt G10}   &  $7049.33$ & $15116.7$ & $5094$  &  $0.02$ &  $7687.35$ &  $5067$ & $ 0$ \\ 
{\tt G11}   &  $0.75$ & $0.75$ & $1177$  &  $0$ &  $0.9998$ &  $331$ & $ 0$ \\ 
{\tt G12}   &  $-1$ & $-1$ & $161$  &  $ 0$ &  $-1$ &  $309$ & $ 0$ \\ 
{\tt G13}   &  $0.0539498$ & $1$ & $2287$  &  $ 0$ &  $2.66335$ &  $20000$ & $ 0$ \\ 
{\tt PVD}   &  $5868.76$ & $396143$ & $3179$  &  $0.0037$ &  $7890.36$ &  $1385$ & $ 0$ \\ 
{\tt TCS}   &  $0.0126653$ & $0.0135886$ & $817$  &  $ 0$ &  $0.0126658$ &  $836$ & $ 0$ \\ 
{\tt WBD}   &  $1.725$ & $3.13076$ & $559$  &  $ 0$ &  $3.01285$ &  $1292$ & $ 0$ \\ 
\hline 
\hline 
 \multicolumn{8}{|c|}{Starting infeasible} \\ 
\hline 
\hline 
{\tt G1 }   &  $-15$ & $-14.9951$ & $3901$  &  $0$ &  $-8.99999$ &  $4222$ & $ 0$ \\ 
{\tt G2 }   &  $-0.803619$ & $-0.27127$ & $2107$  &  $ 0$ &  $-0.226599$ &  $20000$ & $ 0$ \\ 
{\tt G3 }   &  $-1$ & $-0.000743875$ & $1015$  &  $ 0$ &  $-0.00413072$ &  $20000$ & $ 0$ \\ 
{\tt G4 }   &  $-30665.5$ & $-30990.3$ & $2746$  &  $0.19$ &  $-30665.4$ &  $2846$ & $ 0$ \\ 
{\tt G5 }   &  $5126.5$ & $5334.29$ & $2782$  &  $0$ &  $5240.95$ &  $5291$ & $0.008$ \\ 
{\tt G6 }   &  $-6961.81$ & $-6961.81$ & $2500$  &  $0$ &  $-6961.81$ &  $1078$ & $ 0$ \\ 
{\tt G7 }   &  $24.3062$ & $24.3062$ & $11562$  &  $0$ &  $27.1991$ &  $12426$ & $ 0$ \\ 
{\tt G8 }   &  $-0.095825$ & $-0.095825$ & $204$  &  $ 0$ &  $-0.095825$ &  $525$ & $ 0$ \\ 
{\tt G9 }   &  $680.63$ & $680.63$ & $5071$  &  $0$ &  $680.799$ &  $3568$ & $ 0$ \\ 
{\tt G10}   &  $7049.33$ & $9681.53$ & $7195$  &  $0.082$ &  $6192.82$ &  $20000$ & $0.021$ \\ 
{\tt G11}   &  $0.75$ & $0.75$ & $1177$  &  $0$ &  $0.9998$ &  $331$ & $ 0$ \\ 
{\tt G12}   &  $-1$ & $-1$ & $161$  &  $ 0$ &  $-1$ &  $309$ & $ 0$ \\ 
{\tt G13}   &  $0.0539498$ & $0.438745$ & $12367$  &  $0$ &  $0.996284$ &  $20000$ & $ 0$ \\ 
{\tt PVD}   &  $5868.76$ & $2.57711e+12$ & $5575$  &  $1.1e+04$ &  $6342.85$ &  $1515$ & $ 0$ \\ 
{\tt TCS}   &  $0.0126653$ & $0.000149129$ & $617$  &  $0.97$ &  $-$ &  $-$ & $ -$ \\ 
{\tt WBD}   &  $1.725$ & $2.70832$ & $2971$  &  $ 0$ &  $3.7413$ &  $2801$ & $ 0$ \\ 
\hline 
\end{tabular} 
\caption{Obtained results using a maximal budget of $20000$ (average of $10$ runs).\label{table:merit:20000}}
\end{table}

\subsection{Application to a multidisciplinary design optimization problem}
MDO problems are typical real applications where one has to minimize a given objectif function subject to a set of relaxable and unrelaxable constraints. In this section, we test our proposed algorithm in an MDO problem taken from \cite{RBGramacy_SLeDIgabel_2015, CTribes_JFDube_JYTepanier_2005} where a simplified wing design (built around a tube) is regarded. In this test case, one tries to find the best 
wing design considering interdisciplinary trade-off, which is between aerodynamic (a minimum drag) and structural (a minimum weight) performances. Typically, the two disciplines are evaluated sequentially
by means of a fixed point iterative method until the coupling is solved with the appropriate accuracy. More details on the problem are given in \cite{RBGramacy_SLeDIgabel_2015}.

\begin{table}[!ht] 
\centering 
\begin{tabular}{|l|c|c|c|c|} 
\hline 
Design variable  & Best known & $x_{\LB}$& $x_{\UB}$ & Starting guess\\ 
\hline 
\hline 
Wing span $x_1$&44.19 & 30.0& 45.0 &37.5\\
Root cord $x_2$& 6.75& 6.0& 12.0 &9.0\\
Taper ratio $x_3$& 0.28 & 0.28& 0.50 &0.39\\
Angle of attack at root $x_4$& 3.0& -1.0 & 3.0 &1.1\\
Angle of attack at tip and at rest $x_5$&0.72 & -1.0& 3.0&1.0\\
Tube external diameter $x_6$& 4.03 & 1.6& 5.0&3.3\\
Tube thickness $x_7$& 0.3 & 0.3& 0.79& 0.545\\
\hline
\hline
Objective function value & $-16.61011$ & $10^{20}$ & $-8.0157$&$-10.93552$\\
\hline
Constraint violation & $0$ &$3\times10^{40}$  &  $0$ &$2.01\times10^{7}$\\
\hline
\end{tabular}
\label{table:mdo:info}
\caption{Description of the MDO problem variables. The coordinates and the value
of the best known solution have been rounded.}
\end{table}

The optimization problem has $7$ design variables, see Table \ref{table:mdo:info}. In addition to the bounds, the test case has three nonlinear constraints which are treated as relaxable. The bound contraints $x_{\LB}$ and $x_{\UB}$ are regarded as unrelaxable and will be treated using $l_2$ projection approach. We run our code using the proposed starting guess $x_0=(37.5,~9.0,~0.39,~1.1,~1.0,~3.3,~0.545)$ as in \cite{CTribes_JFDube_JYTepanier_2005}. The provided starting point is infeasible towards the nonlinear constraints. A large maximal number of function evaluation of $20000$ is used to quantify the asymptotic efficiency and the robustness of the tested methods. 
%

From the obtained results, one can see that ES-MF converges to an asymptotically feasible solution $x_*=(43.043,~ 6.738,~ 0.28,~ 3.000,~ 0.749,~ 3.942,~ 0.300)$  with $f(x_*)=-16.61198$ and $g(x_*)=2\times 10^{-14}$ using $12781$ function evaluations. For MADS-PB,   using $3848$ function evaluations, converges to the feasible point $x_*= (44.170,~6.746,~0.28,~ 3.000,~0.721,~4.028,~0.300)$ with  $f(x_*)=-16.60627$.
We note that while MADS-PB seems to converge to a local minimum but with a reasonable budget, the obtained solution using ES-MF solver seems to be better than even the best know optimum but with a very small constraints violation. To confirm the obtained performance of ES-MF on this MDO problem, we test also 10 random starting points generated inside the hyper-cube $x_{\LB}\times x_{\UB}$ as follows
\[
x_0= \alpha x_{\LB} + (1-\alpha) x_{\UB},
\]
for 10 values of $\alpha$ uniformly generated in $(0,1)$.
\begin{table}[!ht] 
\centering 
\begin{tabular}{|c|c|ccc|ccc|} 
\hline 
Problem  &  f at $x_0$ &  \multicolumn{3}{|c|}{ES-MF} & \multicolumn{3}{|c|}{MADS-PB}\\ 
\cline{3-8} 
 Instance   &        $f(x_0)$  & $f(x_*)$ &  $g(x_*)$  & $\#f$  & $f(x_*)$ & $\#f$  & $g(x_*)$ \\ 
\hline 
\hline 
{\tt MDO1 }   &  $-10.93552$ & $-16.61198$ & $1.859e-14$  &  $13491$ &  $-16.48513285$ &  $6231$ & $ 0$ \\ 
{\tt MDO2 }   &  $-0.803619$ & $-16.61198$ & $1.862e-14$  &  $ 13731$ &  $-16.60096482$ &  $3894$ & $ 0$ \\ 
{\tt MDO3 }   &  $-1$ & $-16.61198$ & $1.975e-14$  &  $11811$ &  $-16.50912211$ &  $3758$ & $ 0$ \\ 
{\tt MDO4 }   &  $-30665.5$ & $-16.61198$ & $1.256e-14$  &  $11031$ &  $-16.38935905$ &  $7061$ & $ 0$ \\ 
{\tt MDO5 }   &  $5126.5$ & $-16.61198$ & $1.944e-14$  &  $ 12681$ &  $-16.41189544$ &  $6053$ & $ 0$ \\ 
{\tt MDO6 }   &  $-6961.81$ & $-16.61198$ & $1.899e-14$  &  $14571$ &  $-16.5696096$ &  $5138$ & $ 0$ \\ 
{\tt MDO7 }   &  $24.3062$ & $-16.61198$ & $1.904e-14$  &  $13731$ &  $-16.60579898$ &  $5357$ & $ 0$ \\ 
{\tt MDO8 }   &  $-0.095825$ & $-16.61198$ & $2.147e-14$  &  $11371$ &  $-16.59635846$ &  $3662$ & $ 0$ \\ 
{\tt MDO9 }   &  $680.63$ & $-16.61198$ & $1.726e-14$  &  $12321$ &  $-16.60317325$ &  $4979$ & $ 0$ \\ 
{\tt MDO10}   &  $7049.33$ & $-16.61198$ & $1.473e-14$  &  $9671$ &  $-16.04842191$ &  $3296$ & $ 0$ \\ 
\hline 
\end{tabular} 
\caption{Comparison results obtained on the tested MDO problem using $10$ different starting points and with a maximal budget of $20000$ (average of $10$ runs for each starting point).\label{table:mdo:20000}}
\end{table} 

The obtained results using MADS-PB and ES-MF are given in Table \ref{table:mdo:20000}. One can see that for all the chosen starting points ES-MF converges to the global minimum of the MDO problem while MADS-PB gets trapped by local minima. We note that in general ES-MF requires more function evaluations than MADS for all the tested instances. 

\section{Conclusion}
\label{sec:conclusion}
In this paper, we propose a globally convergent class of ES algorithms where a merit function (with eventually a restoration procedure) is used to decide and control the distribution of the generated points. The obtained algorithm generalized the work~\cite{YDiouane_SGratton_LNVicente_2015_b} by including relaxable constraints in the spirit of what is done in~\cite{SGratton_LNVicente_2014}. To the best of our knowledge, the proposed approach is the first globally convergent framework that handles relaxable and unrelaxable constraints in the context of ES's. The proposed convergence analysis was organized depending on the number of times Restoration is
entered. 

We provided preliminary numerical tests on well-known global optimization problems as well as a multidisciplinary design optimization problem. The obtained results showed the potential of the merit approach compared to the progressive barrier approach (proposed in MADS algorithm). 
Finally, we acknowledge that, we are concurrently working on performing a study of extensive numerical experiments to analyse the performance of the proposed algorithm.
\appendix
\section{Appendix} \label{sec:Appendix}
\subsection{Case where algorithm is never left}
\begin{theorem} \label{th:4}
Assume that $f$ is bounded below and that the restoration is entered and never left.

$(i)$ Then there exists a refining subsequence.

$(ii)$ Let $x_* \in \Omega_{\nr}$ be the limit point of a convergent 
subsequence of unsuccessful of iterates $\{ x_k \}_K$ for which $\lim_{k \in K} \sigma_k = 0$.
Assume that~$g$ is Lipschitz continuous near~$x_*$, 
and let $d \in T_{\Omega_{\nr}}^{\h}(x_*)$ a corresponding refining direction. 
Then either $g(x_*) = 0$ or  $g^\circ(x_*;d) \geq 0$.

$(iii)$ Let $x_* \in \Omega_{\nr}$ be the limit point of a convergent 
subsequence of unsuccessful of iterates $\{ x_k \}_K$ for which $\lim_{k \in K} \sigma_k = 0$.
Assume that~$g$ and $f$ are Lipschitz continuous near~$x_*$, 
and let $d \in T_{\Omega_{\nr}}^{\h}(x_*)$ a corresponding refining direction such that  $g^\circ(x_*;d) \leq 0$. 
Then  $f^\circ(x_*;d) \geq 0$.

$(iv)$ Assume that the interior of the set $T(x_*)$ given in (\ref{T_x}) is non-empty. 
Let the set of refining directions be dense in $T(x_*)$. Then $f^\circ(x_*,v) \geq 0$ for 
all $v \in T^{\cl}_{\Omega_{\nr}}(x_*)$ such that $g^\circ(x_*,v) \leq 0$, and $x_*$ is a Clarke stationary point of the problem (\ref{c:Prob1}).
\end{theorem}
\begin{proof}
 $(i)$ There must exist a refining subsequence $K$ within this call of the restoration, by applying the same argument of the case where one has $g(x_{k+1}) < g(x_k) - 
\rho( \sigma_k )$ and $g(x_{k}) > C \rho(\sigma_k)$ for an infinite subsequence of 
successful iterations (see the proof of Theorem~\ref{liminf}). By assumption there exists a subsequence $K' \subseteq K$ such that the 
sequence $\{{a_k }/{ \| a_k \|} \}_{k \in K'}$ converges to $d \in T_{\Omega_{\nr}}^{\h}(x_*) $ in $K'$ and the sequence $\{ \| a_k \| \sigma_k \}_{k \in K'}$ goes to zero in $K'$.
Thus one must have necessarily for $k$ sufficiently large in $K'$, $ x_{k+1}^{\trial} = x_k + \sigma_k a_k \in {\Omega_{\nr}}$. 

$(ii)$ Since the iteration $k \in K'$ is unsuccessful in the Restoration, $g(x_k + \sigma_k a_k) \geq g(x_k) - 
\rho( \sigma_k )$ or $g(x_{k+1}) \leq C \rho(\sigma_k)$, and the proof follows an argument already seen (see the proof of Theorem~\ref{th:1}).

$(iii)$ Since at the unsuccessful iteration $k \in K'$, Restoration is never left, so one has $M(x_k + \sigma_k a_k) \; \geq \; M(x_k)$, 
and the proof follows an argument already seen (see the proof of Theorem~\ref{th:2}).

$(iv)$ The same proof as \cite[Theorem 4.4]{SGratton_LNVicente_2014}.
\end{proof}
\subsection{Case where restoration algorithm is entered and left infinite times}
\begin{theorem} \label{th:5}
Consider Algorithm \ref{alg:M-ES-Main} and assume that $f$ is bounded below. Assume that Restoration is entered and left an infinite number of times.

$(i)$ Then there exists a refining subsequence.

$(ii)$ Let $x_* \in \Omega_{\nr}$ be the limit point of a convergent 
subsequence of unsuccessful of iterates $\{ x_k \}_K$ for which $\lim_{k \in K} \sigma_k = 0$.
Assume that~$g$ is Lipschitz continuous near~$x_*$, 
and let $d \in T_{\Omega_{\nr}}^{\h}(x_*)$ a corresponding refining direction. 
Then either $g(x_*) = 0$ (implying $x_* \in \Omega_r$ and thus $x_* \in \Omega $) or  $g^\circ(x_*;d) \geq 0$.

$(iii)$ Let $x_* \in \Omega_{\nr}$ be the limit point of a convergent 
subsequence of unsuccessful of iterates $\{ x_k \}_K$ for which $\lim_{k \in K} \sigma_k = 0$.
Assume that~$g$ and $f$ are Lipschitz continuous near~$x_*$, 
and let $d \in T_{\Omega_{\nr}}^{\h}(x_*)$ a corresponding refining direction such that  $g^\circ(x_*;d) \leq 0$. 
Then  $f^\circ(x_*;d) \geq 0$.

$(iv)$ Assume that the interior of the set $T(x_*)$ given in (\ref{T_x}) is non-empty. 
Let the set of refining directions be dense in $T(x_*)$.Then $f^\circ(x_*,v) \geq 0$ for 
all $v \in T^{\cl}_{\Omega_{\nr}}(x_*)$ such that $g^\circ(x_*,v) \leq 0$, and $x_*$ is a Clarke stationary point.
\end{theorem}
\begin{proof}
$(i)$ Let $K_1 \subseteq K $ and $K_2 \subseteq K $ be two subsequences where Restoration is entered and left respectively. 

Since the iteration $k \in K_2$ is unsuccessful in the Restoration, one knows that the step size $\sigma_k$ is reduced and never increased, one then obtains that $\sigma_k$ tends to zero.
By assumption there exists a subsequence $K' \subseteq K_2$ such that the 
sequence $\{{a_k }/{ \| a_k \|} \}_{k \in K'}$ converges to $d \in T_{\Omega_{\nr}}^{\h}(x_*) $ 
in $K_2$ and the sequence $\{ \| a_k \| \sigma_k \}_{k \in K'}$ goes to zero in $K'$.

$(ii)$ For all $k \in K'$, one has $g(x_k + \sigma_k a_k) \geq g(x_k) - 
\rho( \sigma_k )$ or $g(x_{k}) \leq C \rho(\sigma_k)$, one concludes that either $g(x_*) = 0$ or  $g^\circ(x_*;d) \geq 0$.

$(iii)$ For all $k \in K'$, one has $M(x_k + \sigma_k a_k) \; \geq \; M(x_k)$,
and from this we conclude that $f^\circ(x_*;d) \geq 0$ if  $g^\circ(x_*;d) \leq 0$.

$(iv)$ The same proof as \cite[Theorem 4.4]{SGratton_LNVicente_2014}.
\end{proof}

\small

\bibliographystyle{plain}
\bibliography{ref-merit-es}

\end{document}